\documentclass[a4paper,12pt]{article}
\usepackage[utf8]{inputenc}
\usepackage[T1]{fontenc}
\usepackage{amsmath,amssymb,amsfonts,amsthm,amstext}
\usepackage{bm}
\usepackage[numbers]{natbib}
\usepackage[scr]{rsfso}
\usepackage{dsfont}
\usepackage{tikz}
\usepackage{nicefrac} 
\usepackage{bbm}

\scrollmode

\numberwithin{equation}{section}

\newtheorem{teo}{Theorem}
\newtheorem{prop}{Proposition}[section]
\newtheorem{lem}{Lemma}[section]
\newtheorem{rmk}{Remark}[section]

\renewcommand{\a}{\alpha}
\newcommand{\A}{\mathcal{A}}

\newcommand{\ba}{\bar{\a}}
\newcommand{\bma}{\bm{\a}}
\newcommand{\bmo}{\bm{1}}

\newcommand{\Bd}{\mathcal{B}}

\newcommand{\bs}{\beta_{\ast}}
\newcommand{\bt}{\beta}
\newcommand{\bsg}{\bar{\sigma}}
\newcommand{\cPhi}{\hat{\Phi}}
\newcommand{\D}{\mathcal{D}}
\renewcommand{\d}{\mathrm{dist}}
\newcommand{\dH}{\text{d}}
\newcommand{\dt}{\delta}
\newcommand{\Dt}{\Delta}
\newcommand{\E}{{\mathbbm E}}
\newcommand{\Es}{\mathcal{E}_N}
\newcommand{\e}{\eta}
\newcommand{\eps}{\epsilon}
\newcommand{\F}{\mathscr{F}}
\newcommand{\fka}{\mathfrak{a}}

\newcommand{\fkm}{\mathfrak{m}}
\newcommand{\fkM}{X}
\newcommand{\fkN}{Y}
\newcommand{\fkw}{\mathfrak{w}}
\newcommand{\G}{\Gamma}
\newcommand{\g}{\gamma}
\newcommand{\Gd}{\mathcal{G}}
\newcommand{\gr}{\mathring{\g}}
\renewcommand{\H}{\mathscr{H}}
\newcommand{\h}{\mathcal{H}}

\newcommand{\I}{\mathcal{I}^{\star}}
\renewcommand{\L}{\mathcal{L}}
\renewcommand{\l}{\ell}
\newcommand{\lb}{\bar{\l}}

\newcommand{\lbd}{\lambda}
\newcommand{\N}{{\mathbbm N}}

\renewcommand{\O}{\Omega}
\renewcommand{\o}{\omega}

\newcommand{\ox}{\overline{x}}
\newcommand{\p}{{\mathbbm P}}
\newcommand{\pas}{\p\mbox{-a.s.}}
\newcommand{\Pt}{\text{P}}
\newcommand{\R}{{\mathbbm R}}

\newcommand{\T}{\mathcal{T}}
\newcommand{\St}{S^{\star}}
\newcommand{\Sg}{\Sigma}
\newcommand{\s}{\sigma}
\renewcommand{\t}{\tau}

\renewcommand{\u}{\upsilon}
\newcommand{\ux}{\underline{x}}

\newcommand{\lc}{\left\lceil}
\newcommand{\rc}{\right\rceil}
\newcommand{\lf}{\left\lfloor}
\newcommand{\rf}{\right\rfloor}
\renewcommand{\lg}{\left\langle}
\newcommand{\rg}{\right\rangle}
\newcommand{\1}{{\mathbbm 1}}

\newcommand{\Nst}{K^{\star}}
\newcommand{\Vs}{N^{\star}}
\newcommand{\ps}{p^{\star}}

\newcommand{\qs}{q^{\star}}
\newcommand{\cs}{c_{\star}}
\newcommand{\rs}{\rho^{\star}}

\usepackage{lmodern}
\usepackage{titling}
\title{\textsc{\Large Convergence Time to Equilibrium of the Metropolis dynamics for the GREM}}

\author{\textsc{\large A. M. B. Nascimento}\thanks{Partially supported by CNPq grant 140762/2016-7}\thanksgap{1ex}\thanksmark{3} \and \textsc{\large L. R. Fontes}\thanks{Partially
supported by CNPq grant 311257/2014-3, and FAPESP grant 2017/10555-0}\thanksgap{1ex}\thanks{Instituto de Matemática e Estat\'istica, Universidade de São Paulo, Rua do Matão 1010, Cidade Universitária, 05508-090 São Paulo SP, Brasil. Emails: {amarcos, lrenato}@ime.usp.br}}

\date{}

\begin{document}

\maketitle

\vspace{-1cm}

\begin{abstract}
We study the convergence time to equilibrium of the Metropolis dynamics for the Generalized Random Energy Model with an arbitrary number of hierarchical levels, a finite and reversible continuous-time Markov process, in terms of the spectral gap of its transition probability matrix. This is done by deducing bounds to the inverse of the gap using a Poincaré inequality and a path technique. We also apply convex analysis tools to give the bounds in the most general case of the model.
\end{abstract}

\vspace{.5cm}
\vspace{.5cm}


\vspace{.5cm}




\section{Introduction and Main Result}
The Generalized Random Energy Model (GREM) is a mean field model for a spin glass in equilibrium, introduced in \cite{D}. Let us describe it. Consider a system with configuration space being $\Sg_N=\{-1,+1\}^N$, the discrete hypercube in $N$ dimensions, equipped with the following hierarchical structure in levels. Fix a number $k\in\N$, such that $k\leq N$, to indicate the number of levels. Let $\{p_j\}_{j=1}^k$ be a sequence of positive real numbers such that $\sum_{j=1}^k p_j=1$ and consider the following partition of the number $N$ into $k$ integers:
\begin{align}\label{Vj}
 N_j=\lf p_j N\rf,\, 1\leq j\leq k-1, \quad \text{and}\quad N_k=N-\sum_{j=1}^{k-1} N_j. 
\end{align}
With this notation, we represent $\Sg_N$ as the product
\begin{align}
 \Sg_N=\Sg_{N_1}\times\cdots\times\Sg_{N_k}
\end{align}
so that a spin configuration $\s\in\Sg_N$ is labeled as $\s=(\s_1,\ldots,\s_k)$ where $\s_j\in\Sg_{N_j}=\{-1,+1\}^{N_j}$ stands for the $j$-th level of $\s$. We denote with $\s^i$ and $\s_j^{i}$ generic spin coordinates of $\s$ and $\s_j$ respectively.

Now, we will define GREM's Hamiltonian on $\Sg_N$. 
Let
$$\H=\H_N=\left\{E^{(j)}_{\s_1\cdots\s_j}: \s_j\in \Sg_{N_j},\,1\leq j\leq k\right\}$$
be a family of independent (vectors of independent) Gaussian random variables of mean 0 and variance $N$.
%
We may view $\H$ 
as a {\it random environment} for the spin model to be defined next. 
Let $\{a_j\}_{j=1}^k$ be a collection of strictly positive real numbers such that $\sum_{j=1}^k a_j=1$, and denote by $\fka$ the vector $\fka=(\sqrt{a_j}: 1\leq j\leq k)$. The 
GREM Hamiltonian on $\Sg_N$ is then defined by
\begin{align}\label{H}
 \h(\s)=-\lg\fka,E_{\s}\rg=-\sum_{j=1}^k\sqrt{a_j}E_{\s_1\cdots\s_j}^{(j)},\, \s\in\Sg_N,
\end{align}
where for each $\s\in\Sg_N$, we denote by $E_{\s}$ the vector $E_{\s}=(E^{(j)}_{\s_1\cdots\s_j}:\,1\leq j\leq k)$, and $\lg \cdot,\cdot\rg$ is the usual inner product on $\R^k$. 
Then $\h=\{\h(\s),\,\s\in\Sg_N\}$ 
is a family of Gaussian random variables with marginal mean zero and variance $N$, and we remark that $\h(\s)$ and $\h(\t)$ are independent if and only if $\s,\t\in\Sg_N$ differ on the first level, i.e., if and only if $\s_1\neq \t_1$.

We denote by $\pi_N$ the 
Gibbs measure at inverse temperature $\bt>0$ associated to the GREM Hamiltonian $\h$ that assigns to each $\s\in\Sg_N$ the mass
\begin{align}\label{GBS}
 \pi_N(\s)=\pi_{k,N,\bt}(\s)=\frac{1}{Z_N}\exp\left(-\bt\h(\s)\right),
\end{align}
where $Z_N\equiv Z_{k,N}(\bt)$ denotes the usual normalizing factor. 
As usual, the function
\begin{align}
 F_N(\bt) =F_{k,N}(\bt)= -\frac{1}{N} \log Z_{k,N}(\bt)
\end{align}
indicates the finite volume free energy. Notice that all those quantities are random variables 
on $(\O,\F,\p)$.

\paragraph{Existence of the Free Energy.} An important equilibrium feature of the GREM that will be needed here is the existence of the free energy: for all $\bt>0$ the limit
\begin{align}\label{FrEn}
 F(\bt)\equiv\lim_{N\uparrow \infty} F_N(\bt)
\end{align}
exists $\p$-almost surely and coincides with $\lim_{N\uparrow \infty} \E(F_N(\bt))$ --- 
see \cite{Picco}, Theorem 2.1. Notice that $F(\bt)$ is a nonrandom function. 

For the sake of completeness, we recall here the explicit formula of $F(\bt)$. To get to that, we start by considering the $k$-dimensional Euclidean space equipped with the norm $\|\cdot\|^2=\lg \cdot,\cdot\rg$. Let us denote by $\Psi_k$ the following subset of $\R^k$,
\begin{align}
 \Psi_k=\left\{x\in\R^k:\sum_{i=1}^j x_i^2 \leq  \bs^2P_j,\,1\leq j\leq k \right\},
\end{align}
where 
$$P_j= \sum_{i=1}^j p_i\,\mbox{ and }\,\bs=\sqrt{2 \log 2}.$$ 
Now, set $J_0^*=0$ and recursively, define
\begin{align}
 J_{l}^*=\min\{J>J_{l-1}^*: B(J_{l-1}^*+1,J)\leq B(J_{l-1}^*+1,j),\forall j\geq J_{l-1}^*+1\},
\end{align}
where $B(i,j)=\bs\sqrt{\frac{p_i+\cdots+p_j}{a_i+\cdots+a_j}}$ for $1\leq i\leq j \leq k$. 
Let $l_k\in\{1,\ldots,k\}$ be such that $J_{l_k}^*=k$. Consider now the collection $(\bt_{l})_{l=0}^{l_k+1}$, where
\begin{align}\label{defb}
 \bt_{l}=B(J_{l-1}^*+1,J_{l}^*),\,1\leq l \leq l_k,
\end{align}
and $\bt_0= 0$ and $\bt_{l_k+1}= \infty$. From the definition of $(J_{l}^*)_{l=1}^{l_k}$, it is clear that 
$(\bt_{l})_{l=1}^{l_k}$ is strictly increasing in $l$. Suppose $\bt\in[\bt_{l},\bt_{l+1})$ for some $0\leq l\leq l_k$, and let $\fkw\equiv\fkw(\bt)\in\Psi_k$ be such that
\begin{align}
 \begin{aligned}
  \fkw_j&=\bt_i\sqrt{a_j},&&\text{if }j\in\{J_{i-1}^*+1,\ldots,J_i^*\}\mbox{ for some }i=1,\ldots,l;\\
 &=\bt\sqrt{a_j},&&\text{if }j\in\{J_{l}^*+1,\ldots,k\}.
 \end{aligned}
\end{align}
With this terminology, $\fkw$ is the point of $\Psi_k$ at minimal distance from
\begin{align}
 \fkm^*\equiv \fkm^*(\bt)=\bt\fka.
\end{align}
We finally have, for all $\bt>0$, that
\begin{align}\label{Fk}
 \begin{aligned}
  F(\bt)&=\frac12(\bs^2+\|\fkm^*\|^2-\|\fkm^*-\fkw\|^2)&&\\
  &=\bt\,\, \sum_{i=1}^{l}\,\,\bt_i \!\!\! \sum_{j=J_{i-1}^* + 1}^{J_i^*}\!\! a_j + 
  \frac12\sum_{j=J_{l}^*+1}^k (\bs^2 p_j + \bt^2 a_j),&&
 \end{aligned}
\end{align}
if $\bt_{l}\leq\bt <\bt_{l+1}$ for some $l=0,\ldots l_k$ --- see \cite{Picco}.
We remark that this function is once, but not twice, continuously differentiable with respect to $\bt$. From a physical point of view, this means that there exist (possibly multiple) third-order phase transitions for the GREM. 
Let us also point out that for $\bt\geq \bt_{l_k}$ 
there exists a unique point $\fkw^*\in\Psi_k$, independent of $\bt$, such that $\fkw=\fkw^*$ and
\begin{align}\label{IL5}
 F(\bt)=\lg\fkm^*,\fkw^*\rg=\max_{x\in \Psi_k} \lg\fkm^*,x\rg.
\end{align}
The latter identity is shown in Appendix, Lemma \ref{L5}.

\paragraph{Dynamics.}
Here, we consider a dynamics for the GREM, that is, we construct a continuous time Markov chain with state space $\Sg_N$, for which the Gibbs measure $\pi_N$ is invariant; indeed, the chain and the GREM are in detailed balance. 
In fact, we consider the Metropolis dynamics. Let us define it next. 
Let us consider the continuous-time Markov process $\{\omega_N(t):t\geq 0\}$, taking values in $\Sg_N$ and having transition probability matrix $\Pt$ with entries given by
\begin{align}\label{MTPL}
 \Pt(\s,\t)=\begin{cases}
             \frac{1}{N} \exp\left(-\bt\left[\h(\t)-\h(\s)\right]^+\right), & \text{if }\dH(\s,\t)=1;\\
             1-\sum_{\e\neq\s}\Pt(\s,\e),&\text{if }\s=\t;\\
             0,&\text{otherwise.}
            \end{cases}
\end{align}
where $\h$ is the GREM Hamiltonian defined in \eqref{H}; $\bt>0$ is the inverse of temperature parameter; $\dH(\cdot,\cdot)$ denotes the usual Hamming distance on $\Sg_N$ and $x^+=x\vee 0$, $x\in\R$. 
This process is reversible, and therefore, both stationary and ergodic, with respect to the Gibbs measure $\pi_N$.

Before discussing our results, let us recall the related results derived for the REM under Metropolis (which corresponds to the GREM with $k=1$).

The following result is implied by Theorem 1 in~\cite{fontes1998spectral}. Let $\lbd_N^{\text{\tiny REM}}$ be the spectral gap of the generator of the dynamics (or, equivalently, of the one-step transition probability matrix). Then for all $\bt>0$ we have that
%
%
\begin{align}\label{fikp}
 \lim_{N\uparrow \infty} -\frac{1}{N} \log \lbd_N^{\text{\tiny REM}} =\bs\bt \quad \pas
\end{align}
Indeed Theorem 1 in~\cite{fontes1998spectral} provides estimates for the errors of approximation that hold a.s.~for all large enough $N$, but we will not be concerned with those here.

In this paper we will derive upper bounds for the analogue in our dynamics of the quantity whose limit is taken in~(\ref{fikp}). These, as is well known, provide upper bounds for the time to reach equilibrium under the dynamics. Let us describe the relevant quantities more precisely.


Let $1=\mu_{N,0}>\mu_{N,1}\geq \cdots \geq \mu_{N,2^N}>-1$ denote the eigenvalues of the one-step transition probability matrix $\Pt$ whose entries are defined in \eqref{MTPL}; since $\Pt$ is reversible with respect to $\pi_N$, we have that 
\begin{align}
 \lbd_N\equiv \lbd_N(\bt)=1-\mu_{N,1}
\end{align}
is its spectral gap. Notice that, in the case of the REM, $\lbd_N=\lbd_N^{\text{\tiny REM}}$. The main result of this paper is the following.
\begin{teo}\label{Thm1}
 For all $\bt>0$,
 \begin{align}\label{IThm1}
  \limsup_{N \uparrow \infty} -\frac{1}{N} \log \lbd_N \leq \lg\fkm^*,\fkw^*\rg\quad \pas
 \end{align}
\end{teo}

Some remarks follow:
\begin{enumerate}
  \item First of all, notice that the bound in the right-hand side of \eqref{IThm1}, viewed as function of $\bt$, is the function that describes the free energy of the GREM for $\bt\geq\bt_{l_k}$. As expected, we get Proposition 4.2 in \cite{fontes1998spectral} as corollary of the Theorem \ref{Thm1} by taking $k=1$. We still remark that Theorem \ref{Thm1} holds for {\it all } $\bt>0$, for {\it all } $k\in\N$ and for {\it any} choice of parameters $\{a_j\}_{j=1}^k$ and $\{p_j\}_{j=1}^k$ satisfying $0<a_j,p_j<1$ and $\sum_{j=1}^k a_j=\sum_{j=1}^k p_j=1$.
  
  \item In view of Theorem \ref{Thm1}, using the following well known bound (see \cite{Diaconis} for a derivation): for all $\s\in\Sg_N$ and $t>0$,
  \begin{align}
   4\left\|\Pt_t(\s,\cdot)-\pi_N(\cdot)\right\|_{\text{\tiny var}}^2\leq \frac{1-\pi_N(\s)}{\pi_N(\s)}\exp\left(-2 \lbd_N t\right),
  \end{align}
  together with \eqref{FrEn} and Theorem 1.5(iii) of \cite{Bovier1}, one deduces that for any $t>\lg\fkm^*,\fkw^*\rg$,
  \begin{align}\label{IThm12}
    \lim_{N\uparrow \infty}\max_{\s}\left\| \Pt_{e^{Nt}}(\s,\cdot)-\pi_N(\cdot)\right\|_{\mathrm{var}}=0,\quad \p\mbox{-a.s.}
  \end{align}
  Here $\Pt_t(\s,\t)=e^{-t}\sum_{n=0}^{\infty}(t^n/ n!) \Pt^n(\s,\t)$ is the transition kernel of the dynamics.
  
  \item There is reason to believe that the bound \eqref{IThm1} is not sharp, based on the results of
  \cite{fontes2018asymptotic}, where 
  large volume limits for a hierarchical, simplified version of the present dynamics are derived for the 
  2 level GREM at low temperature (in the {\em cascading} phase), with time properly scaled. The limit dynamics are ergodic processes, and have the (infinite volume) Gibbs measure as equilibrium measure. The time scalings for those results are always below what is implied by \eqref{IThm1}, and this would indicate that the latter bound is not sharp (at least at low temperatures). 

On the other hand, under the dynamics of \cite{fontes2018asymptotic}, it may be proved that \eqref{IThm1} is the best bound one gets (to leading order) by using the Poincaré inequality employed in the present work (at all temperatures).

\item A direct analysis of the Metropolis dynamics for the GREM at time scales where one would expect to see an ergodic large volume limiting dynamics, as has been done in \cite{fontes2018asymptotic} for a simpler dynamics, has not been undertaken yet; even for the $k=1$ case of the REM, this has been done only at smaller time scales, where aging takes place instead ---
see \cite{vcerny2017aging,gayrard2016} --- and, indeed, spectral gap estimations are important elements in the derivations. 

See also \cite{Arous2018} for applications of spectral gap estimation on the study of a class of dynamics for a large family of mean field spin glasses.
 \end{enumerate} 

The rest of the paper is devoted to prove Theorem \ref{Thm1}. 
In Section \ref{Sec-Thm1}, we develop our bound to the inverse of the spectral gap, in terms of the canonical path approach by Jerrum and Sinclair.
This leads to the statement of two propositions which immediately lead to the proof of Theorem \ref{Thm1}. The proof of the first of 
the propositions is done in Section \ref{Sec-PropfkM1}, in several steps which take most of the remainder of the paper. Section \ref{Sec-PropfkM2} contains the similar, shortly presented proof of the second proposition, and an appendix is devoted to supporting results.

%

   
\section{Proof of the Theorem \ref{Thm1} -- Canonical set of paths}\label{Sec-Thm1} 
As mentioned above, the proof of Theorem \ref{Thm1} relies on a Poincaré inequality derived in~\cite{sinclair1991improved}.
To write this inequality in our context, the first step is to identify the Markovian process $\omega_N(t)$ with an undirected graph with vertex set $\Sg_N$. Naturally, we identify it with the $N$-dimensional hierarchical hypercube graph which we will also denote, with a little abuse, by $\Sg_N$. Let us denote $\Es=\{(\s,\t)\in \Sg_N^2: \dH(\s,\t)=1\}$ the edge set of $\Sg_N$. Now, let $\G_N=\{\g_{\e\u}:\e,\u\in\Sg_N\}$ be a complete set of self-avoiding canonical paths on $\Sg_N$, that is, for each $\e,\u\in\Sg_N$, there exists exactly one path $\g_{\e\u}$ in $\G_N$ connecting $\e$ and $\u$ using only valid transitions of the Markov chain $\omega_N(t)$, that is, only through edges of $\Es$. Denote by $\lb=\lb(\G_N)$ the maximum length of paths (i.e. number of edges) in $\G_N$. Then, from Theorem 5 in \cite{sinclair1991improved} we have
\begin{align}\label{cg0}
 \frac{1}{\lbd_N} \leq \varrho(\G_N)=\max_{e=(\s,\t)} \left\{\frac{\lb}{\pi_N(\s)\Pt(\s,\t)}\sum_{\g_{\e\u}\ni e} \pi_N(\e)\pi_N(\u)\right\}
\end{align}
where the maximum is over all edges $e=(\s,\t)\in \Es$ and the summation is over all pairs $(\e,\u)$ such that there exists a path $\g_{\e\u}$ in $\G_N$ that contains edge $e$. The expression $\varrho(\G_N)$ is 
called the {\it congestion} associated with the set of paths $\G_N$. Recall \eqref{GBS} and \eqref{MTPL}. Using them, it is easy to check that
\begin{align}\label{cg1}
 \varrho(\G_N) = \frac{\lb N}{Z_N}\max_{e=(\s,\t)} \left\{\exp\left(\bt[\h(\s)\vee \h(\t)]\right) \sum_{\g_{\e\u}\ni e} \exp\left(-\bt[\h(\e)+\h(\u)]\right)\right\}.
\end{align}
Notice that to apply efficiently inequality in \eqref{cg0} we need now to construct a suitable set of paths $\G_N$ that allows us to get a good upper bound to $\varrho(\G_N)$. By ``good'', we mean that on the limit, in the very spirit of \eqref{IThm1}, such bound coincides $\p$-almost surely with $\lg \fkm^*,\fkw^*\rg$.

When one tries to obtain a spectral gap estimate for the Metropolis dynamics of spin glass models using the canonical path technique, one of the first concerns is with edges $e=(\s,\t)\in\Es$ where 
$\h(\s)\vee \h(\t)$ is {\it large}. A natural attempt to control these {\it bad} edges is to avoid them as much as possible in the trajectories. The completeness of $\Gamma_N$ implies that they cannot be avoided as extreme edges of paths, but we may try to 
avoid them in the {\em interior} of paths; as we will see below, we succeed in doing that {\it with high probability}, with a set of paths that is amenable enough to subsequent analysis. This approach was already used in \cite{fontes1998spectral}.
%
%
%
Observe that with such set of paths, if $e=(\s,\t)\in\g_{\e\u}$ is a bad edge, then we have that either $\s=\e$ and $\t$ has the lowest energy, or $\s$ has lowest energy and $\t=\u$. Considering the first case --- the other one follows by symmetry ---, the term inside of the $\max$ sign in \eqref{cg1} can be estimated by
\begin{align}\label{cg10}
 \exp\left(\bt\h(\s)\right) \sum_{\g_{\s\u}\ni e} \exp\left(-\bt[\h(\s)+\h(\u)]\right) = \sum_{\u \neq \s} \exp\left(-\bt\h(\u)\right) \leq Z_N.
\end{align}
%

To construct our suitable set of paths $\G_N$, we need to introduce some notation. Let $\kappa>0$ be arbitrary. 
We say that a configuration $\s\in\Sigma_N$ is {\it good} if $$\h(\s)\leq \kappa N;$$ otherwise, we will call it {\it bad}.
We will call any set of configurations, in particular an edge of $\Es$, good if all the configurations in it are good; otherwise, we will call the set bad.
 Then the set $\Es$ can be written as the following disjoint union: $\Es=\Gd\cup\Bd$, where $\Gd$ and $\Bd$ denote the sets of good and bad edges, respectively. 
 
 For any path $\g=\{e_1,e_2,\ldots,e_n\}$ with $e_j\in\Es$, $j=1,\ldots,n$, let $\gr=\{e_2,e_3,\ldots,e_{n-1}\}$ denote the set of {\it interior edges} of $\g$. A path $\g$ with all interior edges good is called good; a set of paths with all elements good is also called good. At this point, it is clear that the set of paths that we aim to construct, a good one, will depend on the realization of the random environment $\H$ which implies that $\G_N$ will be a random set of paths.

One of the fundamental concepts we will need here is the notion of independent paths. Two paths $\g_1$ and $\g_2$ 
will be called {\it independent} if for all $\s\in\gr_1$ and $\t\in\gr_2$, the random variables $\h(\s)$ and $\h(\t)$ are independent; equivalently, if $\s_1\neq \t_1$. An extension of this concept for a finite family of paths in $\Sg_N$ can be done in an obvious way. At last, let us denote by $\dH_1(\cdot,\cdot)$, resp.~$\dH(\cdot,\cdot)$, the usual Hamming distance on $\Sg_{N_1}$, resp.~$\Sg_N$.

With these concepts in hands, we have the following lemma where we specify one condition under which there exist independent paths connecting 
configurations in $\Sigma_N$.
This will also motivate our subsequent definition of $\G_N$. 
\begin{lem}\label{Lindep}
 Let $\e$ and $\u$ be two configurations in $\Sg_N$. If $\dH_1(\e_1,\u_1)=n\geq 2$, then there exists a family containing $n$ independent paths connecting $\e$ to $\u$.
\end{lem}

\begin{proof}
Consider, for each pair of distinct vertices $\e,\u\in\Sg_N$, the set of paths
\begin{align}
 \G(\e,\u)=\{\g_{\e\u}^i:i=1,2,\ldots,N\},
\end{align}
where $\g_{\e\u}^i$ denotes the path from $\e$ to $\u$ defined as follows. Suppose $\dH(\e,\u)=r\geq n$; then let 
$1\leq \l_{m+1}<\cdots<\l_r<i\leq \l_1<\cdots<\l_m\leq N$ be the positions where $\e$ and $\u$ disagree, 
$m\in\{0,\ldots, r\}$. Let $\g_{\e\u}^i$ be the path starting at $\e$ and ending at $\u$ whose 
$j$-th edge, $1\leq j\leq r$, corresponds to flipping 
$\e_{\l_j}$ to $\u_{\l_j}$. 

For future reference, we set 
\begin{equation}\label{Gi}
\G^i=\{\g_{\e\u}^i:\e,\u\in\Sg_N\},\,i=1,2,\ldots,N. 
\end{equation}

We will now argue that $\G(\e,\u)$ is a family of paths that satisfies the required property.
Let $1\leq i_1<\cdots<i_n\leq N_1$ be the positions where $\e$ and $\u$ disagree on the first level, and consider the set of paths $\{\g_{\e\u}^{i_1},\ldots,\g_{\e\u}^{i_n}\}$. We claim that this set of paths is independent. Indeed, this is quite clear if the discrepancies between $\e$ and $\u$ are only in the first level. Otherwise, let us first notice that it is enough to consider the case where $\e_1$ and $\u_1\equiv+1$ differ in the $n$ first coordinates (where thus $\e_1\equiv-1$); now it is just a matter of noticing that any interior configuration $\s$ of $\g_{\e\u}^{i_j}$ is characterized by the condition that $\s_1^{i_{j-1}}=-1$ and  $\s_1^{i_j}=+1$ (in this paragraph, $i_0$ should be understood as $i_n$).

\end{proof}

With the help of this lemma, we can now construct the random set of paths that we will consider in \eqref{cg1}. Let $0<\eps <\nicefrac{1}{2}$ be arbitrary:
\begin{enumerate}
 \item For a given pair of distinct configurations $\e$ and $\u$ such that $\dH_1(\e_1,\u_1)\geq \eps N_1$, if there exists a good path in $\G(\e,\u)$, then we choose one such path, say the one with the smallest superscript, for $\G_N$; otherwise, we choose $\g_{\e\u}^1$;
 \item 
 If $\dH_1(\e_1,\u_1)<\eps N_1$, and there exists a good vertex $\s'\in\Sg_N$ such that $\dH_1(\e_1,\s_1')\geq\eps N_1$, $\eps N_1\leq\dH_1(\s_1',\u_1)=\dH(\s',\u)\leq 2\eps N_1$ and there exist good paths, one in $\G(\e,\s')$ and another in $\G(\s',\u)$, such that the concatenation of these two paths is a self-avoiding path with length less than $N$, then we choose this concatenation as the path from $\e$ to $\u$ in $\G_N$ (notice that this is a good path since $\s'$ is good); otherwise, we choose  $\g_{\e\u}^1$.\label{Cond2}
\end{enumerate}

It is immediate that $\G_N$ thus chosen 
is a complete set of self-avoiding paths, that is each pair $\e,\u\in\Sg_N$ is uniquely connected by a self-avoiding path $\g_{\e\u}\in\G_N$. Moreover, we may readily check that  
$\lb(\G_N)\leq N$, so we get the bound
\begin{align}\label{cg2}
 \varrho(\G_N)\leq \frac{N^2}{Z_N}\max_{e=(\s,\t)} \left\{\exp\left(\bt[\h(\s)\vee \h(\t)]\right) \sum_{\g_{\e\u}\ni e} \exp\left(-\bt[\h(\e)+\h(\u)]\right)\right\}.
\end{align}
The following is a key fact about $\G_N$.
\begin{prop}\label{SetG}
 For any $\kappa>0$ and any $0<\eps <\nicefrac{1}{2}$ the following holds: with $\p$-probability $1$ there exists an $N_0=N_0(\kappa,\eps)\in\N$ such that for all $N\geq N_0$ the set of paths $\G_N$ is good. 
\end{prop}

\begin{proof}
 For pairs of vertices $\e,\u\in\Sg_N$ such that $\dH_1(\e_1,\u_1)\geq \eps N_1$, the $\p$-almost sure existence of good paths connecting them in $\G_N$ is proved arguing as Proposition 4.1 in \cite{fontes1998spectral} using the help of Lemma \ref{Lindep}.
 
 For pairs of vertices $\e,\u\in\Sg_N$ such that $\dH_1(\e_1,\u_1)<\eps N_1$, let us first denote by $\D_1^{\e,\u}=\{i: \e_1^i\neq\u_1^i\}$ the set of positions where $\e$ and $\u$ differ on the first level and also introduce the set
 \begin{align}\label{SetSgNeu}
  \Sg_N^{\e,\u}=\{\s\in\Sg_N:\s_1|_{\D_1^{\e,\u}}=\e_1|_{\D_1^{\e,\u}},\dH_1(\s_1,\e_1)=\lc \eps N_1 \rc \text{ and }\s_j=\u_j,j=2,\ldots,k\}
 \end{align}
 where the condition ``$\s_1|_{\D_1^{\e,\u}}=\e_1|_{\D_1^{\e,\u}}$'' is not present if $\D_1^{\e,\u}=\varnothing$. Here, $\s_1|_D=(\s^i)_{i\in D}$ is just the restriction of $\s_1$ to set $D\subseteq\{1,\ldots,N_1\}$.
 %
 We may readily check that $\dH_1(\e_1,\o_1)\geq\eps N_1$ and $\eps N_1\leq\dH_1(\o_1,\u_1)=\dH(\o,\u)\leq 2\eps N_1$ for all $\o\in\Sg_N^{\e,\u}$. 
%
  
  For $\s\in\Sg_N^{\e,\u}$, let $\g_{\s\s'}$  stands for the path starting at the vertex $\s$, constructed by flipping the sites whose positions belong to $\D_1^{\e,\u}$, in increasing order of coordinate. In case $\D_1^{\e,\u}=\varnothing$, we assume that $\s=\s'$ and $\g_{\s\s'}=\{\s\}$. By Lemma \ref{lemSg} below, there are at least $(2\eps)^{-\eps N_1}$ such paths, which are independent by construction. Thus, since there exists a constant $c_{\kappa}>0$ such that the probability of all visited vertices for a given such path $\g_{\s\s'}$ to be bad can be bounded by $e^{-c_{\kappa}N}$ when $N$ is large enough, for any $\kappa>0$ and $0<\eps<\nicefrac{1}{2}$, we can found $N'=N'(\kappa,\eps)\in\N$ such that for all $N\geq N'$,
  \begin{align}
   \p\left[\bigcap_{(\e,\u)} \bigcap_{\s\in\Sg_N^{\e,\u}}\{\g_{\s\s'}\text{ is bad}\} \right]\leq \sum_{N\geq N'} 4^N e^{-c_{\kappa}N(2\eps)^{-\eps N_1}}<\infty.
  \end{align}
  It then follows from the Borel-Cantelli Lemma that, for any $\kappa>0$ and $0<\eps<\nicefrac{1}{2}$, with $\p$-probability $1$, for all $N$ sufficiently large there exists at least one vertex, say $\o\in\Sg_N^{\e,\u}$, and its corresponding path, say $\g_{\o\o'}$, which is good. 
  By construction we have that 
  $\e,\o$ are more than distance $\eps N_1$ apart, 
  and so are $\o',\u$; as before, for any $\kappa>0$ and any $0<\eps<\nicefrac{1}{2}$, we can $\pas$ find good paths $\g_{\e\o}$ and $\g_{\o'\u}$ for all $N$ large enough. The conclusion of this case now follows by concatenating the (good) paths $\g_{\e\o},\g_{\o\o'}$ and $\g_{\o'\u}$, to get the path from $\e$ to $\u$ in $\G_N$. 
\end{proof}

\begin{lem}\label{lemSg}
	For any $0< \eps <\nicefrac{1}{2}$ and any $\e,\u\in\Sg_N$ such that $\dH_1(\e_1,\u_1)<\eps N_1$, 
	let $\Sg_N^{\e,\u}$ be as in~(\ref{SetSgNeu}). Then
	\begin{align}
	|\Sg_N^{\e,\u}| \geq (2\eps)^{-\eps N_1}.
	\end{align}
\end{lem}

\begin{proof}
	We have that
	\begin{align}
	|\Sg_N^{\e,\u}|\geq \binom{N_1-\lf \eps N_1 \rf}{\lc \eps N_1 \rc}\geq \left(\frac{1-\eps}{\eps+\frac{1}{N_1}}\right)^{\eps N_1}
	\end{align}
	where the last inequality follows from the fact that $\binom{n}{m} \geq (\nicefrac{n}{m})^m$, $n\geq m\geq 1$, and standard bounds for $\lf \cdot \rf$ and $\lc \cdot \rc$. Now, since $N_1 \uparrow \infty$ as $N \uparrow \infty$, for any $0<\eps<\nicefrac{1}{2}$, we have that $N_1^{-1}\leq \eps-2\eps^2$ for any $N$ sufficiently large. This is enough to get the statement of the lemma.
\end{proof}

Having constructed the set of paths $\G_N$, we can now proceed with the spectral gap estimate. From now on we assume that, for all $\kappa>0$ and all $0<\eps <\nicefrac{1}{2}$, $\pas$ for all large enough $N$, $\G_N$ is good. Recalling that $\Es=\Gd\cup\Bd$, where $\Gd$ and $\Bd$ denote the sets of good and bad edges respectively, we can write
\begin{align}\label{cg20}
 \varrho(\G_N)\leq \frac{N^2}{Z_N}(\fkM_N^{\Gd}\vee \fkM_N^{\Bd}),
\end{align}
where $\fkM_N^{\Gd}$, respectively $\fkM_N^{\Bd}$, is as the maximum term in \eqref{cg2} but with the $\max$ sign restrict to edges in $\Gd$, respectively $\Bd$. From \eqref{cg10}, it follows immediately that $\fkM_N^{\Bd} \leq Z_N$ and, by Proposition \ref{SetG}, one readily concludes that  $\fkM_N^{\Gd} \leq \exp\left(\kappa \bt N\right) \fkM_N$, where
\begin{align}
 \fkM_N=\max_{e\in\Gd}\left\{\sum_{\g_{\e\u}\ni e} \exp\left(-\bt[\h(\e)+\h(\u)]\right)\right\}.
\end{align}
Using these last bounds in \eqref{cg20}, $\varrho(\G_N)$ can be estimated by 
\begin{align}\label{cg200}
 \varrho(\G_N)\leq N^2 \vee \left( N^2Z_N^{-1} \exp\left(\kappa \bt N\right) \fkM_N\right), \quad \pas
\end{align}
for all large enough $N$.

Let now
\begin{align}
 \fkM_N^{(1)}&=\max_{e\in\Gd}\left\{\sum_{\g_{\e\u}\ni e} \exp\left(-\bt[\h(\e)+\h(\u)]\right)\1\{\dH_1(\e_1,\u_1)\geq \eps N_1\}\right\};\\\label{x2n}
 \fkM_N^{(2)}&=\max_{e\in\Gd}\left\{\sum_{\g_{\e\u}\ni e} \exp\left(-\bt[\h(\e)+\h(\u)]\right)\1\{\dH_1(\e_1,\u_1)< \eps N_1\}\right\};
\end{align}
so we have $\fkM_N\leq\fkM_N^{(1)}+\fkM_N^{(2)}$. 

In Sections \ref{Sec-PropfkM1} and \ref{Sec-PropfkM2}, we prove the following two results, respectively.
\begin{prop}\label{PropfkM1}
 For all $\bt>0$,
 \begin{align}
  \limsup_{N\uparrow \infty} \frac{1}{N}\log \fkM_N^{(1)} \leq F(\bt)+\lg \fkm^*,\fkw^*\rg, \quad \pas
 \end{align}
\end{prop}

\begin{prop}\label{PropfkM2}
 For all $\bt>0$,
 \begin{align}
  \limsup_{N\uparrow \infty} \frac{1}{N}\log \fkM_N^{(2)} \leq F(\bt)+\lg \fkm^*,\fkw^*\rg, \quad \pas
 \end{align}
\end{prop}

These propositions, combined with \eqref{FrEn},
immediately yield Theorem \ref{Thm1}.


\section{Proof of Proposition \ref{PropfkM1}}\label{Sec-PropfkM1}
We follow the  strategy  in
\cite{fontes1998spectral} (see Subsection 4.2 therein), with steps that are increasingly more involved than in the $k=1$ case of that reference; in particular, our last two steps depart considerably from the direct approach there.
\paragraph{Step 1 -- Bound in terms of $\G^1,\ldots,\G^N$.}
Since the set $\G_N$ is constructed using paths in $\bigcup_{i=1}^N\G^i$, if we denote
\begin{align}
 M_i=\max_{e\in\Gd}\sum_{\g_{\e\u}^i\ni e} \exp\left(-\bt[\h(\e)+\h(\u)]\right),
\end{align}
for $i=1,\ldots,N$ and $M_{(N)}=M_1\vee \cdots\vee M_N$, we get the estimate
\begin{align}\label{cg21}
 \fkM_N^{(1)}\leq N M_{(N)}.
\end{align}
Since $M_1,\ldots,M_N$ 
are identically distributed, it is sufficient to give an estimate for one of them with a relatively good probability estimate. Consider thus 
 \begin{align}
  M_1=\max_{1\leq i\leq N} \max_{\substack{e=(\s,\t)\\ \s^i\neq \t^i}}\left\{\sum_{\g_{\e\u}^1\ni e} \exp\left(-\bt[\h(\e)+\h(\u)]\right)\right\}.
 \end{align}

For a given edge $e=(\s,\t)$, there exists a unique coordinate $i\in\{1,\ldots,N\}$ such that $\s^i\neq \t^i$. So that, by construction, the set of all pairs $(\e,\u)$ such that $\g_{\e\u}^1\ni e$ is exactly
 \begin{align}
  \left(\bigcup_{\e\in\{-1,+1\}^{i-1}} \{(\e,\s^i,\ldots,\s^N)\} \right)\times \left(\bigcup_{\u\in\{-1,+1\}^{N-i}}\{(\t^1,\ldots,\t^i,\u)\}\right).
 \end{align}
 Then, if we denote $\s^{>i}=(\s^{i+1},\ldots,\s^N)$, $\s^{<i}=(\s^1,\ldots,\s^{i-1})$,
 \begin{align}
  S_{i-1}^{(1)}(\s^i,\s^{>i})=\sum_{\e\in\{-1,+1\}^{i-1}}\exp\left(-\bt\h(\e,\s^i,\s^{>i})\right)
 \end{align}
 and
 \begin{align}
  S_{N-i}^{(1)}(\s^{<i},-\s^i)=\sum_{\u\in\{-1,+1\}^{N-i}}\exp\left(-\bt\h(\s^{<i},-\s^i,\u)\right),
 \end{align}
 we obtain the bound 
  \begin{align}\label{cg3}
  M_1\leq \max_{1\leq i\leq N} \max_{\s\in\Sg_N} S_{i-1}^{(1)}(\s^i,\s^{>i})S_{N-i}^{(1)}(\s^{<i},-\s^i).
 \end{align}
\paragraph{Step 2 -- Coarse graining.}
 Now we will focus on estimating the right-hand side of \eqref{cg3}. Before turning to this, let us briefly describe our strategy. 
 We partition the $k$-dimensional Euclidean space into subsets $\Dt_{\l_1,\ldots,\l_k}$, and
 analyse separately the contribution to 
 $S_{i-1}^{(1)}(\s^i,\s^{>i})$ and $S_{N-i}^{(1)}(\s^{<i},-\s^i)$ coming from each $\Dt_{\l_1,\ldots,\l_k}$,
 by means of large deviation-type estimates, thus securing control over the exponentially many terms involved in the above maximization. 
%
%
 It is enough to study $S_{i-1}^{(1)}(\s^i,\s^{>i})$ in detail; the case of $S_{N-i}^{(1)}(\s^{<i},-\s^i)$ is entirely similar.
 
 Let $1\leq i\leq N$, $\s^i=\pm 1$ and $\s^{>i}\in\{-1,+1\}^{N-i}$ be fixed, 
 and let $j$ be such that $i\in\{1+\sum_{n=1}^{j-1}N_n,\ldots,\sum_{n=1}^j N_n\}$. 
 Let $\a$ be such that $\a N_j=i-(1+\sum_{n=1}^{j-1}N_n)$, and set $\bma^j=(\a_1,\ldots,\a_k)$ such that
 \begin{align}\label{ALPHAs}
  \begin{aligned}
   \a_n&=1,\quad&&\text{if }n<j,\\
        &=\a,\quad&&\text{if }n=j,\\
        &=0,\quad &&\text{if }n>j.
  \end{aligned}
 \end{align}
 Let $\Sg_{r,s}^{\bma^j}=\Sg_{\a_r N_r}\times \cdots\times \Sg_{\a_s N_s}$, $1\leq r\leq s\leq k$. We can thus write
 \begin{align}
  \{-1,+1\}^{i-1}=\Sg_{1,j}^{\bma^j},
 \end{align}
 and 
 \begin{align}
  \{-1,+1\}^{N-i}=\Sg_{j,k}^{\bmo-\bma^j}, 
 \end{align}
 where $\ba_j=1-\a_j$, and $\bmo=(1,\ldots,1)$. We stress the relationship between $i$, $j$ and $\a$ established in this paragraph.
 
 \begin{rmk}
  Notice that if $i\in\{1,N_1,N_1+N_2,\ldots,N\}$ (cases equivalent to $\a\in\{0,1\}$), then 
  we readily get that
  \begin{align}\label{cg4}
   M_1^*\equiv\max_{i\in\{1,N_1,\ldots,N\}}\max_{\s\in\Sg_N} S_{i-1}^{(1)}(\s^i,\s^{>i})S_{N-i}^{(1)}(\s^{<i},-\s^i)\leq \exp\left(-\bt\h(\bsg)\right)Z_N.
  \end{align}
  By Theorem 1.5(iii) in \cite{Bovier1} and \eqref{FrEn}, we thus have that for all $\bt>0$,
  \begin{align}
   \limsup_{N\uparrow \infty} \frac{1}{N}\log M_1^* \leq F(\bt)+\lg \fkm^*,\fkw^*\rg.
  \end{align}
 \end{rmk}


For convenience, we enumerate/represent 
$$
\left\{-E^{(n)}_{\t_1\cdots\t_n};\,n=1,\ldots,k;\,\t\in\{-1,+1\}^{i-1}\times\{\s^i\s^{>i}\}\right\}
$$
as
\begin{equation}\label{rep}
\left\{E_{u_1,\ldots,u_n}^{(n)};\, u_n=1,\ldots,2^{\a_nN_n};\, n=1,\ldots,k\right\}.
\end{equation}

Set $E_u=(E_{u_1}^{(1)},\ldots,E_u^{(k)})$, $u=(u_1,\ldots,u_k)$. 
With this notation, $S_{i-1}^{(1)}(\s^i,\s^{>i})$ can be written as
 \begin{align}
  S_{i-1}^{(1)}(\s^i,\s^{>i})=\sum_{u_1=1}^{2^{\a_1 N_1}}\cdots\sum_{u_j=1}^{2^{\a_j N_j}} \exp\left(\lg\fkm^*,E_u\rg\right).
 \end{align}
 
 Let $L\in\N$ and 
 consider the following partition of $\R^k$:
 $$
  \R^k=\bigcup_{\l_1=0}^{L+1}\cdots\bigcup_{\l_k=0}^{L+1} \Delta_{\l_1,\ldots,\l_k},\quad \text{with }\Delta_{\l_1,\ldots,\l_k}=\Delta_{\l_1}^1\times\cdots\times\Delta_{\l_k}^k,
 $$
 where for $n=1,\ldots,k$, we set
 \begin{align}
  \begin{aligned}
   \Delta_{\l_n}^n &=\left(-\infty,\frac1L\bs \sqrt{P_n}N\right],\quad &&\text{if }\l_n=0;\\
                   &=\left( \frac{\ell_n}{L}\bs\sqrt{P_n}N,\frac{\ell_n+1}{L}\bs\sqrt{P_n}N\right],\quad&&\text{if }\l_n=1,\ldots,L;\\
                   &=\left(\left(1+\frac1L\right) \bs\sqrt{P_n}N,\infty \right),\quad&&\text{if }\l_n=L+1.
  \end{aligned}
 \end{align}
 Now we decompose $S_{i-1}^{(1)}(\s^i,\s^{>i})$ in the following way:
 \begin{align}\label{B1}
  \begin{aligned}
   S_{i-1}^{(1)}(\s^i,\s^{>i})=\sum_{\l_1,\ldots,\l_k=0}^L&\left(\sum_{u_1=1}^{2^{\a_1 N_1}}\cdots\sum_{u_j=1}^{2^{\a_j N_j}}\1\{E_u\in\Delta_{\l_1,\ldots,\l_k}\}\right)\exp\left(\lg \fkm^*,E_u\rg\right)\\
   &\quad +\sum_{u_1=1}^{2^{\a_1 N_1}}\cdots\sum_{u_j=1}^{2^{\a_j N_j}}\left(\sum_{\L^*}\1\{E_u\in\Delta_{\l_1,\ldots,\l_k}\}\right)\exp\left(\lg \fkm^*,E_u\rg\right),
  \end{aligned}
 \end{align}
 where $\L^*:=\{0\leq \l_1,\ldots,\l_k\leq L+1:\exists\, n\in\{1,\ldots,k\}\mbox{ such that } \l_n=L+1\}$.
 
 First we consider the last sum in the right-hand side of \eqref{B1}; denote it by $S_N^*$. 
 We will show that this quantity is zero  for all $N$  large enough $\pas$ Indeed, we note first that
 \begin{align}\label{sns}
  S_N^*\leq \exp\left(-\h(\bsg)\right)\sum_{u_1=1}^{2^{N_1}}\cdots\sum_{u_k=1}^{2^{N_k}}\1\{E_u\in\cup_{\L^*}\Delta_{\l_1,\ldots,\l_k}\}.
 \end{align}
 Now consider the event
 \begin{align}
  \A_{L,N}=\left\{\forall j=1,\ldots,k, \forall u_1,\ldots,u_k, \sum_{i=1}^j (E_{u_1,\ldots,u_i}^{(i)})^2 \leq \left(1+\frac1L\right)P_jN\right\}.
 \end{align}
 One readily checks that $\left\{\sum_{u_1}\cdots\sum_{u_k}\1\{E_u\in\cup_{\L^*}\Delta_{\l_1,\ldots,\l_k}\}\geq 1\right\}\subset \A_{L,N}^c$, so that, from Proposition 3.1 in \cite{Picco}, we have that the sum in~(\ref{sns}),
 and thus  $S_N^*$,
 vanishes for all large $N$ $\pas$ 

\paragraph{Step 3 -- Large deviation estimate.}
 
 It remains to bound the first term in the right-hand side of \eqref{B1}. In order to do this, we need to introduce some notation. Given $0\leq r \leq s \leq k$, define the canonical projection $\Pi_r^s\colon\R^k\to\R^{s-r}$  such that
 $\Pi_r^s x=(x_{r+1},\ldots,x_s)$, where 
 by convention $\Pi_s^s\equiv 0$. Set $\Psi_r^s\equiv\Pi_r^s\Psi_k=\{\Pi_r^s x:x\in\Psi_k\}$. Now, let $\Phi_r^s:\Psi_r^s\to[0,\infty)$ be the functional defined by
 \begin{align}\label{firs}
  x \mapsto \Phi_r^s(x)=\lg \Pi_r^s\fkm^*,x\rg.
 \end{align}
 We remark that by compactness and convexity, $\Phi_r^s$ admits a unique maximum on $\Pi_r^s\Psi_k$, at say $z_r^s\in\partial (\Pi_r^s\Psi_k)$; 
 set $\cPhi_r^s=\Phi_r^s(z_r^s)=\lg \Pi_r^s\fkm^*,z_r^s\rg$. 
 We note that $\|z_r^s\|^2=\bs^2\sum_{m=r+1}^s p_m$.
 
 For each $0\leq r\leq s\leq k$, let us set
 \begin{align}\label{dQ}
     Q_r^s\equiv Q_r^{s,\bma^j}=\sum_{m=r+1}^s \a_m p_m,
 \end{align}
 with the convention that $Q_r^r\equiv 0$.
 
Let now $\ux=(\ux_1,\ldots,\ux_k)\in\R^k$ be such that $\ux_n=\frac{\l_n}L\bs\sqrt{P_n}$, $n=1,\ldots,k$.
 With the above terminology, we have that $\pas$ for all $N$ large enough,
 \begin{align}\label{cg450}
  \begin{aligned}
   S_{i-1}^{(1)}(\s^i,\s^{>i}) \leq e^{\tfrac{\bs^2}{2}\,Q_0^j N} e^{\frac kL \bs\bt N}+ S_j^N
  \end{aligned}
 \end{align}
 where
 \begin{align}\label{cg45}
  S_j^N=L^{k-j}\exp\left(\left[\frac kL\bs\bt + \cPhi_j^k\right]N\right)\sum_{n=1}^j\sum_{[i_1,\ldots,i_n]}\sum_{\l_{i_1},\ldots,\l_{i_n}=1}^L K_{\l_{i_1},\ldots,\l_{i_n}} \exp\left(\lg \Pi_0^j\fkm^*,\Pi_0^j \ux \rg N\right),
 \end{align} 
 with the middle sum above 
 being over all sequences of integers $1\leq i_1<\cdots<i_n \leq j$, and
 \begin{align}\label{Nl0}
  \begin{aligned}
   K_{\l_{i_1},\ldots,\l_{i_n}}&\equiv K_{\l_{i_1},\ldots,\l_{i_n}}(1,i,\s^i,\s^{>i})\\
   & =e^{\tfrac{\bs^2}{2}\,Q_{i_n}^jN} \sum_{u_1=1}^{2^{\a_1 N_1}}\cdots\sum_{u_{i_n}=1}^{2^{\a_{i_n} N_{i_n}}}\prod_{r=1}^n \1\{E_{u_1,\ldots,u_{i_r}}^{(i_r)}\in\Delta_{\l_{i_r}}^{i_r}\},
  \end{aligned}
 \end{align}
 
 \begin{rmk}\label{rmk2}
  Note that in \eqref{cg45} the point $\Pi_0^j \ux=(\ux_1,\ldots,\ux_j)$ is such that $\ux_r=0$ for all $r\neq i_1,\ldots,i_n$.
 \end{rmk}
 
 Let $n\in\{1,\ldots,j\}$, $[i_1,\ldots,i_n]$ and $1\leq \l_{i_1},\ldots,\l_{i_n}\leq L$ be fixed. For $1\leq r\leq n$, set
 \begin{align}\label{dVr}
  \Vs_r=\sum_{s=i_{r-1}+1}^{i_r} \a_s N_s;\quad 
  \ps_r=\sum_{s=i_{r-1}+1}^{i_r} \a_s p_s;\quad
 \end{align}
 where $i_0=0$. 
 %
%
 With this notation, we 
 write
 \begin{align}\label{Nl1}
  \begin{aligned}
   K_{\l_{i_1},\ldots,\l_{i_n}}&=e^{\tfrac{\bs^2}{2}\,Q_{i_n}^jN}\sum_{u_1=1}^{2^{\Vs_1}}\cdots\sum_{u_n=1}^{2^{\Vs_n}}\prod_{r=1}^n \1\{E_{u_1,\ldots,u_r}^{(i_r)}\in\Delta_{\l_{i_r}}^{i_r}\}:= e^{\tfrac{\bs^2}{2}\,Q_{i_n}^jN} \Nst_{\l_{i_1},\ldots,\l_{i_n}}.
  \end{aligned}
 \end{align}
 
 Now let us estimate $\Nst_{\l_{i_1},\ldots,\l_{i_n}}$. Let $\qs_{\l_{i_r}}=\p(E_{u_1,\ldots,u_{i_r}}^{(i_r)}\in \Dt_{\l_{i_r}}^{i_r})$. 
 We then have for all $r=1,\ldots,n$ and $N$ large enough that
 \begin{align}\label{iq}
  e^{-\tfrac{1}{2}\, \ox_{i_r}^2 N}\leq \qs_{\l_{i_r}} \leq e^{-\tfrac{1}{2}\, \ux_{i_r}^2N},\quad\ox_{i_r}=\frac{\l_{i_r}+1}L\bs\sqrt{P_{i_r}}.
 \end{align}
 
 Let now $\cs>0$ be a positive constant to be specified later, and define the following family of integers. For all $1\leq r\leq s \leq n$, set $U_{r,s}=\prod_{m=r}^s \qs_{\l_{i_m}} 2^{\Vs_m}$. Let $J_0=0$ and recursively define
\begin{align}
 J_{\nu}=\max\{s:\,J_{\,\nu-1}<s\leq n:U_{J_{\nu-1}+1,s}<\cs N\}
\end{align}
until $\nu=\nu_n\in\{0,\ldots,n\}$ such that $J_{\nu_n}=n$ or $U_{J_{\nu_n}+1,s}\geq \cs N$ for all $J_{\nu_n}+1 \leq s \leq n$. Put $J_{\nu_n+1}=n+1$. We then have that $0=J_0<J_1<\cdots<J_{\nu_n}<J_{\nu_n+1}= n+1$. Moreover, for every $\nu=0,\ldots,\nu_n$,
  \begin{align}\label{irmk3}
   U_{J_{\nu}+1,s}\geq \cs N, \quad \forall s\in\{J_{\nu}+1,\ldots,J_{\nu+1}-1\}.
  \end{align}
At last, if $\nu_n=0$, then put
 \begin{align}\label{drs}
  \rs_{\l_{i_r}}=4 \quad \text{for all } r=1,\ldots,n;
  \end{align}
  otherwise, that is, if $\nu_n\in\{1,\ldots,n\}$, then put
  \begin{align}
   \begin{aligned}
    \rs_{\l_{i_r}}&=4\cs N U_{J_{\nu-1}+1,J_{\nu}}^{-1},\quad&&\text{if }r=J_{\nu}\text{ for some }\nu=1,\ldots,\nu_n;\\
     &=4,\quad&&\text{if }r\neq J_1,\ldots,J_{\nu_n}.
   \end{aligned}
  \end{align}

 \begin{lem}\label{LemNst}
  With the notation introduced above, for any $\cs>0$ and $1\leq \l_{i_1},\ldots,\l_{i_n}\leq L$, the following holds for all 
  large enough $N$.
  \begin{align}\label{dLemNst}
   \p \left(\Nst_{\l_{i_1},\ldots,\l_{i_n}} > \prod_{r=1}^n \rs_{\l_{i_r}}\qs_{\l_{i_r}} 2^{\Vs_r}\right)\leq n e^{-\cs N}.
  \end{align}
 \end{lem}
 
 In the proof of Lemma~\ref{LemNst}, we will use the following result.
  \begin{lem}\label{LemNst1}
   For each $1\leq n\leq j$, let $b_n=\prod_{r=1}^n \rs_{\l_{i_r}}\qs_{\l_{i_r}} 2^{\Vs_r}$. 
   Set 
   $$b_0\equiv 1\mbox{ and }\vartheta_n \equiv \left(\rs_{\l_{i_n}}-1\right)b_{n-1}\qs_{\l_{i_n}}2^{\Vs_n}.$$ 
   Then
   \begin{align}\label{Fact1}
    \vartheta_n  \geq 3 \cs N.
   \end{align}
  \end{lem}
  
  \begin{proof}
   From the definition of the $\rs_{\l_{i_r}}$ and \eqref{irmk3}, it follows that
   \begin{align}
    \rs_{\l_{i_r}}\geq 4 \quad \text{and}\quad b_n=4^n (\cs N)^{\nu_n}U_{J_{\nu_n}+1,n}\geq 4\cs N,
   \end{align}
   and thus $\vartheta_n\geq 3b_n/4 \geq 3\cs N$.
  \end{proof}
  
  \begin{proof}[Proof of the Lemma \ref{LemNst}]
   We will argue by induction on $n$. In the case $n=1$, since $\rs_{\l_{i_1}}\geq 4$, by Chernoff's inequality and Lemma \ref{LemNst1} above, we have
   \begin{align}
     \p\left(\Nst_{\l_{i_1}} >b_1\right)=\p\left(\Nst_{\l_{i_1}} >\rs_{\l_{i_1}}\qs_{\l_{i_1}}2^{\Vs_1}\right)\leq e^{-\frac{1}{3}\vartheta_1 }\leq e^{-\cs N}.
   \end{align}
   Assume that \eqref{dLemNst} is proved for $n-1$. Introducing the random set $$I_{n-1}=\{(u_1,\ldots,u_{n-1}):E_{u_1,\ldots,u_r}^{(i_r)}\in\Dt_{\l_{i_r}}^{i_r},\forall r=1,\ldots,n-1\}$$ 
   and taking into account the independence of the Gaussian random variables, we may write   
   \begin{align}\label{dLemNst2}
    \begin{aligned}
     \p\left(\Nst_{\l_{i_1},\ldots,\l_{i_n}}>b_n\right)
     &\leq \p\left(\Nst_{\l_{i_1},\ldots,\l_{i_{n-1}}}>b_{n-1}\right)\\
     &+\p\left(\sum_{I_{n-1}}\sum_{u_n=1}^{2^{\Vs_n}}\1\{E_{u_1,\ldots,u_n}^{(i_n)}\in\Dt_{\l_{i_n}}^{i_n}\}> b_n \middle | \Nst_{\l_{i_1},\ldots,\l_{i_{n-1}}}\leq b_{n-1}\right)\\
     &\leq (n-1) e^{-\cs N}+\p\left(\sum_{u_0=1}^{b_{n-1}2^{\Vs_n}} \1\{E_{u_0}\in\Dt_{\l_{i_n}}^{i_n}\}> b_n\right),    
    \end{aligned}
   \end{align}
   where we also use on the second inequality the induction hypothesis \eqref{dLemNst} for $n-1$; here, $\{E_{u_0}\}$ is a relabeling of the random variables $\{E_{u_1,\ldots,u_n}^{(i_n)}\}$. It remains to bound the last term on the right-hand side of \eqref{dLemNst2}. Notice that $b_n=\rs_{\l_{i_n}}b_{n-1}\qs_{\l_{i_n}}2^{\Vs_n}$. Since $\rs_{\l_{i_n}}\geq 4$, it follows from Chernoff's inequality and Lemma \ref{LemNst1} above that
   \begin{align}
    \p\left(\sum_{u_0=1}^{b_{n-1}2^{\Vs_n}} \1\{E_{u_0}\in\Dt_{\l_{i_n}}^{i_n}\}> \rs_{\l_{i_n}}b_{n-1}\qs_{\l_{i_n}}2^{\Vs_n}\right) \leq e^{-\frac{1}{3} \vartheta_n}\leq e^{-\cs N}.
   \end{align}
   This concludes the proof. 
  \end{proof}

 Coming back to \eqref{cg3}, we need to make a probability estimate which holds for all possible random variables $\Nst_{\l_{i_1},\ldots,\l_{i_n}}$ involved in the max signs. Recall that there is an index for the chosen path family, the index $i$, the configurations $\s^i$, $\s^{>i}$, and the indices $n$, $[i_1,\ldots,i_n]$ and $\l_{i_1},\ldots,\l_{i_n}$. Since there are not more than $2L^kN^22^{N+k}$ distinct such objects, it suffices to have a probability estimate in \eqref{dLemNst} to compensate for this factor. This suggests the choice of $\cs$ for the following result,
 which is immediate from Lemma~\ref{LemNst} and the union bound.  
 \begin{prop}\label{C1}
  Given $\dt>0$, assume that $\cs>\log 2 +2\dt$. Then for all $N$ sufficiently large,
  \begin{align}\label{sjn}
   \p\left( \exists \; \Nst_{\l_{i_1},\ldots, \l_{i_n}}> \prod_{r=1}^n\rs_{\l_{i_r}}\qs_{\l_{i_r}} 2^{\Vs_r}\right) \leq e^{-\dt N}. 
  \end{align}
 \end{prop}
 
 In view of \eqref{cg45} and \eqref{Nl1}, one readily deduces from~(\ref{sjn}) that for any given $\dt>0$, with $\p-$probability $\geq 1-e^{-\dt N}$, for all $N$ large enough,
    \begin{align}\label{cg5}
   \begin{aligned}
    S_j^N\leq L^{k-j}\exp\left(\left[\frac kL\bs\bt + \cPhi_j^k\right]N\right) 
    \sum_{n=1}^j\sum_{[i_1,\ldots,i_n]} \T^{(n)}_{i_1,\ldots,i_n},
   \end{aligned}
  \end{align}
  where
     \begin{align}\label{cg6}
     \T^{(n)}_{i_1,\ldots,i_n}=
     e^{\tfrac{\bs^2}{2}\,Q_{i_n}^jN} \sum_{\l_{i_1},\ldots,\l_{i_n}=1}^L \left(\prod_{r=1}^n \rs_{\l_{i_r}}\qs_{\l_{i_r}} 2^{\Vs_r}\right) \exp\left(N\lg \Pi_0^j\fkm^*,\Pi_0^j \ux \rg\right).
  \end{align}

 The next step is to estimate (the non-random term on) the right hand side of~(\ref{cg6}).

 \paragraph{Step 4 -- Deterministic estimation.}

 It is worth noticing at this point that
 we have to make our estimation uniform with respect to all the indices involved.
 
 Let $n\in\{1,\ldots,j\}$ and $[i_1,\ldots,i_n]$ be fixed. We partition the support of the sum in \eqref{cg6} 
 into the subsets
 \begin{align}
  \I_{n,j}(s)=\I_{n,j}(s)[i_1,\ldots,i_n]=\{1\leq \l_{i_1},\ldots,\l_{i_n}\leq L : J_{\nu_n}=s\}, \quad s=0,\ldots,n.
 \end{align}
 If $\l_{i_1},\ldots,\l_{i_n}\in\I_{n,j}(s)$, then we have from \eqref{iq} that
 \begin{align}
  \prod_{r=1}^n \rs_{\l_{i_r}}\qs_{\l_{i_r}} 2^{\Vs_r} \leq (4 \cs N)^j \exp\left(\tfrac{1}{2}N\sum_{r=s+1}^n(\bs^2 \ps_r-\ux_{i_r}^2)\right)
 \end{align}
 for all $N$ large enough, where the exponential factor is not present if $s=n$. \eqref{cg6} can thus be bounded above by
 \begin{align}\label{cg7}
  (4 \cs N)^j \sum_{s=0}^n \sum_{\I_{n,j}(s)} \exp\left(\tfrac{1}{2}\left[\bs^2Q_{i_n}^j+\1\{s\neq n\}\sum_{r=s+1}^n(\bs^2 \ps_r-\ux_{i_r}^2)+ 2\lg \Pi_0^j\fkm^*,\Pi_0^j \ux \rg\right]N\right).
 \end{align}
 
 For $0\leq s\leq n$, let 
 \begin{align}
  \St_{n,j}(s)=\sum_{\I_{n,j}(s)} \exp\left(\tfrac{1}{2}\left[\bs^2Q_{i_n}^j+\1\{s\neq n\}\sum_{r=s+1}^n(\bs^2 \ps_r-\ux_{i_r}^2)+ 2\lg \Pi_0^j\fkm^*,\Pi_0^j \ux \rg\right]N\right)
 \end{align}
 We will estimate $\St_{n,j}(s)$ by distinguishing the case $s\in\{0,\ldots,n-1\}$ from the case $s=n$.
 
 \medskip
 
 \noindent{\bf Case I:} $s\in\{0,\ldots,n-1\}$. 
 From \eqref{dQ}, \eqref{dVr} and Remark \ref{rmk2}, we get that
 \begin{align}\label{cg61}
  \bs^2Q_{i_n}^j+\sum_{r=s+1}^n(\bs^2 \ps_r-\ux_{i_r}^2)=\bs^2Q_{i_s}^j-\lg \Pi_{i_s}^j\ux,\Pi_{i_s}^j \ux \rg.
 \end{align}
 Now, from basic properties of inner product, we can also write
 \begin{align}\label{cg62}
  2\lg \Pi_0^j\fkm^*,\Pi_0^j \ux \rg-\lg \Pi_{i_s}^j\ux,\Pi_{i_s}^j \ux \rg=2\lg \Pi_0^{i_s}\fkm^*,\Pi_0^{i_s} \ux \rg+\left\|\Pi_{i_s}^j\fkm^*\right\|^2-\left\|\Pi_{i_s}^j\fkm^*-\Pi_{i_s}^j \ux\right\|^2.
 \end{align}
 Moreover, by construction, if $\l_{i_1},\ldots,\l_{i_n}\in\I_{n,j}(s)$, then $\Pi_0^{i_s} \ux\in \Psi_0^{i_s}$,
  and thus
 \begin{align}\label{cg63}
  \eqref{cg62}\leq 2 \cPhi_0^{i_s}+\left\|\Pi_{i_s}^j\fkm^*\right\|^2-\left\|\Pi_{i_s}^j\fkm^*-\Pi_{i_s}^j \ux\right\|^2.
 \end{align}
 Thus, by suitably using \eqref{cg61} and \eqref{cg63}, we get that 
 \begin{align}\label{cg64}
  \St_{n,j}(s)\leq \exp\left(\cPhi_0^{i_s}N\right)\sum_{\I_{n,j}(s)} \exp\left(\tfrac{1}{2}\left[\bs^2Q_{i_s}^j+\left\|\Pi_{i_s}^j\fkm^*\right\|^2-\left\|\Pi_{i_s}^j\fkm^*-\Pi_{i_s}^j \ux\right\|^2\right]N\right).
 \end{align}
 
 Now, for $0\leq r \leq s\leq k$, set 
 \begin{align}
  \Psi_{r,s}^{\bma^j}=\left\{x\in\R^{s-r}:\forall m=1,\ldots,s-r, \sum_{n=1}^m x_n^2 \leq \bs^2\sum_{n=r+1}^{r+m} \a_n p_n\right\},
 \end{align}
 and note that $\Psi_{r,s}^{\bma^j}$ is a nonempty closed convex subset of $\R^{s-r}$ so, from Theorem \ref{PBrezis} (see Appendix \ref{AP1}), there exists a unique element of $\Psi_{r,s}^{\bma^j}$, say $\fkw_{r,s}^{\bma^j}$, such that
 \begin{align}\label{dist}
  \d(\Pi_r^s\fkm^*,\Psi_{r,s}^{\bma^j})=\left\|\Pi_r^s\fkm^*-\fkw_{r,s}^{\bma^j}\right\|.
 \end{align}
 
 Since we have $s=J_{\nu_n}$, it is not difficult to see with the help of Remark \ref{rmk2} and \eqref{irmk3} that $\Pi_{i_s}^j \ux\in \Psi_{i_s,j}^{\bma^j}$; it then follows from~\eqref{dist} that 
 \begin{align}\label{cg65}
  \left\|\Pi_{i_s}^j\fkm^*-\Pi_{i_s}^j \ux\right\|^2\geq \left\|\Pi_{i_s}^j\fkm^*-\fkw_{i_s,j}^{\bma^j}\right\|^2.
 \end{align}
 
 For each $0\leq l\leq r\leq k$, let
 \begin{subequations}
  \begin{alignat}
   2 G_{l,r}(\bt,\bma^j)&=\frac{\bs^2}2 Q_l^{r,\bma^j}+
   \frac12\left(\left\|\Pi_l^r\fkm^*\right\|^2-\left\|\Pi_l^r\fkm^*-\fkw_{l,r}^{\bma^j}\right\|^2\label{Gnj1}\right)\\
   &=\frac{\bs^2}2 Q_l^{r,\bma^j}+ \lg\Pi_l^r\fkm^*,\fkw_{l,r}^{\bma^j}\rg - \frac12\lg\fkw_{l,r}^{\bma^j},\fkw_{l,r}^{\bma^j}\rg.\label{Gnj2}
  \end{alignat}
 \end{subequations}
 With this definition, \eqref{cg64} and \eqref{cg65} imply that for all $N$ large enough and for any $s=0,\ldots,n-1$,
 \begin{align}
  \St_{n,j}(s)\leq L^j \exp\left(N\left[\cPhi_0^{i_s}+G_{i_s,j}(\bt,\bma^j)\right]\right).
 \end{align}
 
 \medskip

 Before going to the next case, let us point out that 
  for any $0\leq l \leq r \leq k$, we have
  \begin{align}\label{qg}
   \bs^2 Q_l^{r,\bma^j} \leq 2 G_{l,r}(\bt,\bma^j).
  \end{align}
 Indeed,
  let $\L_{l,r}^{\bma^j}\colon\Psi_{l,r}^{\bma^j}\to \R$ be given by $\L_{l,r}^{\bma^j}(x)=\lg\Pi_l^r\fkm^*-\fkw_{l,r}^{\bma^j},x-\fkw_{l,r}^{\bma^j}\rg$, and notice, from  \eqref{Gnj2}, 
  that 
  \begin{align}
   \bs^2 Q_l^{r,\bma^j}-2 G_{l,r}(\bt,\bma^j)=\lg\fkw_{l,r}^{\bma^j},\fkw_{l,r}^{\bma^j}\rg-2\lg\Pi_l^r\fkm^*,\fkw_{l,r}^{\bma^j}\rg \leq \L_{l,r}^{\bma^j}(0).
  \end{align}
 According to Theorem \ref{PBrezis} (see Appendix \ref{AP1}), we have that $\L_{l,r}^{\bma^j}(x)\leq 0$ for all $x\in\Psi_l^{r,\bma^j}$. The claim follows.

 \noindent{\bf Case II:} $s=n$. In this case, since
 \begin{align}
  \lg \Pi_0^j\fkm^*,\Pi_0^j \ux \rg=\lg \Pi_0^{i_n}\fkm^*,\Pi_0^{i_n} \ux \rg\leq \cPhi_0^{i_n},
 \end{align}
 it is immediate from~(\ref{qg}) that $\St_{n,j}(n)$ can be estimated by
 \begin{align}
 \St_{n,j}(n)\leq L^j \exp\left(N[\cPhi_0^{i_n}+G_{i_n,j}(\bt,\bma^j)]\right)
 \end{align}
 for all large enough $N$. Summarizing and coming back to \eqref{cg7}, we get that
 \begin{align}\label{cg8}
  \begin{aligned}
   \eqref{cg6}&\leq c(LN)^j \exp\left(N{\textstyle\bigvee}_{s=0}^n\left[\cPhi_0^{i_s}+G_{i_s,j}(\bt,\bma^j)\right]\right)\\
   &\leq c(L N)^j \exp\left(N {\textstyle\bigvee}_{s=0}^j \left[\cPhi_0^s+G_{s,j}(\bt,\bma^j)\right]\right)
  \end{aligned}
 \end{align}
 for all $N$ sufficiently large, for some $c>0$ not depending on $N$ or $L$, where $\cPhi_0^0 \equiv G_{s,s}\equiv 0$.
   
 Recall 
 \eqref{cg5}. In view of \eqref{cg8} and standard combinatorial estimates, we obtain that for any $\dt>0$, with a $\p-$probability $\geq 1-e^{-\dt N}$ for all $N$ large enough,
  \begin{align}\label{cg66}
   S_j^N\leq c(LN)^k e^{\frac kL \bs\bt N } \exp\left(N \cPhi_j^k + N {\textstyle\bigvee}_{s=0}^j \left[\cPhi_0^s+G_{s,j}(\bt,\bma^j)\right]\right)
  \end{align}
  for some constant $c>0$. This concludes the estimation of $S_j^N$. 
  
  \medskip
  
  Let us now recall 
  \eqref{cg450}. Since we have already estimated $S_j^N$, it remains to estimate the term 
  $e^{\frac kL \bs\bt N }e^{\tfrac{\bs^2}{2}\,Q_0^j N}$. From~(\ref{qg}), 
  we readily find that
  \begin{align}\label{cg67}
   e^{\tfrac{\bs^2}{2}\,Q_0^j N}\leq 
   \exp\left( G_{0,j}(\bt,\bma^j )N\right).
  \end{align}
  
  It follows from Proposition \ref{C1}, \eqref{cg66} and \eqref{cg67} that for any $\dt>0$, with a $\p-$probability $\geq 1-e^{-\dt N}$ for all $N$ large enough,
  \begin{align}
   \max_{\s^i,\s^{>i}} S_{i-1}^{(1)}(\s^i,\s^{>i})\leq c (LN)^k e^{\frac kL \bs\bt N } 
   \exp\left\{\left(\cPhi_j^k + N {\textstyle\bigvee}_{s=0}^j \left[\cPhi_0^s+G_{s,j}(\bt,\bma^j)\right]\right)N\right\}.
  \end{align}
  Symmetrically, we also have that
  \begin{align}
   \max_{\s^{<i},-\s^i} S_{N-i}^{(1)}(\s^{<i},-\s^i)\leq c(LN)^k e^{\frac kL \bs\bt N }
   \exp\left\{\left( \cPhi_0^{j-1} + {\textstyle\bigvee}_{r=j-1}^k \left[\cPhi_{j-1}^r+G_{r,k}(\bt,\bmo-\bma^j)\right]\right)N\right\}.
  \end{align}
  
  Thus, letting 
  \begin{align}\label{psij}
   \psi_j(\bt,\bma^j)=\cPhi_0^{j-1}+\cPhi_j^k + {\textstyle\bigvee}_{s=0}^j {\textstyle\bigvee}_{r=j-1}^k \left[\cPhi_0^s+\cPhi_{j-1}^r+G_{s,j}(\bt,\bma^j)+G_{r,k}(\bt,\bmo-\bma^j)\right],
  \end{align}
  we get that with $\p-$probability $\geq 1-e^{-\dt N}$ for all $N$ large enough,
  \begin{align}\label{maxSs}
   \max_{\s\in\Sg_N} S_{i-1}^{(1)}(\s^i,\s^{>i})S_{N-i}^{(1)}(\s^{<i},-\s^i)\leq c^2 (LN)^{2k} 
   e^{2\frac kL \bs\bt N }\exp\left( \psi_j(\bt,\bma^j) N\right).
  \end{align}
  
  \paragraph{Step 5 -- Maximization.}

  As a final step, 
  it remains to maximize  $\psi_j(\bt,\bma^j)$ over $j\in\{1,\ldots,k\}$ and $\a\in [0,1]$. We do this in the following lemma.
  
  \begin{lem}\label{Lempsij}
   For every $1\leq j\leq k$ and all $0\leq \a\leq 1$,
   \begin{align}
    \psi_j(\bt,\bma^j)\leq \lg \fkm^*,\fkw^* \rg +F(\bt).
   \end{align}
  \end{lem}
  
  \begin{proof}
   Recall the definitions \eqref{ALPHAs} of $\bma^j$, \eqref{dQ} of $Q_r^{s,\bma^j}$ and (\ref{Gnj1}-b) of $G_{r,s}(\bt,\bma^j)$. Also recall, from the discussion at the beginning of Step 3 above, that $\cPhi_r^s=\lg \Pi_r^s\fkm^*,z_r^s \rg$ denotes the maximum of $\Phi_r^s$ over $\Psi_r^s$, attained at point $z_r^s$. 
   We claim that
   \begin{align}\label{cg9}
     \cPhi_0^s+\cPhi_{j-1}^r+G_{s,j}(\bt,\bma^j)+G_{r,k}(\bt,\bmo-\bma^j)\leq \cPhi_{j-1}^j+F(\bt),
   \end{align}
   for any $0\leq s \leq j$ and any $j-1 \leq r \leq k$.
   
   We check this for $0\leq s< j$ and $r=j-1$. The other cases follows from similar arguments. Noting that $\cPhi_{j-1}^r$ is not present in the left-side of \eqref{cg9}, from definition of $\bma^j$, it is equal to 
   \begin{align}\label{cg91}
       \begin{aligned}
          \lg \Pi_0^s\fkm^*,z_0^s\rg + \frac{\bs^2}{2}\sum_{n=s+1}^k p_n &+ \lg\Pi_s^j\fkm^*,\fkw_{s,j}^{\bma^j}\rg - \frac12\lg\fkw_{s,j}^{\bma^j},\fkw_{s,j}^{\bma^j}\rg\\
          &+ \lg\Pi_{j-1}^k\fkm^*,\fkw_{j-1,k}^{\bmo-\bma^j}\rg - \frac12\lg\fkw_{j-1,k}^{\bmo-\bma^j},\fkw_{j-1,k}^{\bmo-\bma^j}\rg.
       \end{aligned}
   \end{align}
   Now, using ``$\circ$'' to indicate vector concatenation, it is immediate to observe that $\fkw_{j-1,k}^{\bmo-\bma^j}=\Pi_0^1 \fkw_{j-1,k}^{\bmo-\bma^j}\circ \Pi_1^{k-j+1}\fkw_{j-1,k}^{\bmo-\bma^j}$, and so we find that the expression in~\eqref{cg91} equals
   \begin{align}\label{cg92}
       \begin{aligned}
          & \,\fkm_j^*\cdot \Pi_0^1 \fkw_{j-1,k}^{\bmo-\bma^j} +\frac{\bs^2}{2}\sum_{t=s+1}^k p_t+ \lg\fkm^*,z_0^s\circ\fkw_{s,j}^{\bma^j}\circ\Pi_1^{k-j+1}\fkw_{j-1,k}^{\bmo-\bma^j}\rg \\
          &\qquad\qquad\qquad - \frac12\lg\fkw_{s,j}^{\bma^j}\circ\Pi_1^{k-j+1}\fkw_{j-1,k}^{\bmo-\bma^j},\fkw_{s,j}^{\bma^j}\circ\Pi_1^{k-j+1}\fkw_{j-1,k}^{\bmo-\bma^j}\rg - \frac 12(\Pi_0^1 \fkw_{j-1,k}^{\bmo-\bma^j})^2\\
          \leq&\, \cPhi_{j-1}^j + \frac{\bs^2}{2} +
          \frac12\left\|\fkm^*\right\|^2-\frac12\left\|\fkm^*-z_0^s\circ\fkw_{s,j}^{\bma^j}\circ\Pi_1^{k-j+1}\fkw_{j-1,k}^{\bmo-\bma^j}\right\|^2,
       \end{aligned}
   \end{align}
   where the inequality follows from the the facts that $\lg z_0^s,z_0^s\rg=\bs^2 \sum_{n=1}^s p_n$ (as noted right below~\eqref{firs}) and $\fkm_j^*\cdot \Pi_0^1 \fkw_{j-1,k}^{\bmo-\bma^j}\leq \cPhi_{j-1}^j$ (by the maximality of the latter quantity). 
   Convexity now implies that $z_0^s\circ\fkw_{s,j}^{\bma^j}\circ\Pi_1^{k-j+1}\fkw_{j-1,k}^{\bmo-\bma^j}\in\Psi_k$, and thus from Theorem \ref{PBrezis} we may conclude that
   \begin{align}
    \left\|\fkm^*-z_0^s\circ\fkw_{s,j}^{\bma^j}\circ\Pi_1^{k-j+1}\fkw_{j-1,k}^{\bmo-\bma^j}\right\|\geq\left\|\fkm^*-\fkw\right\|=\d(\fkm^*,\Psi_k).
   \end{align}
   \eqref{cg9} is now just a matter of recalling \eqref{Fk}.
   
   From \eqref{cg9}, we find that 
   \begin{align}
    \psi_j(\bt,\a)\leq \cPhi_0^{j-1}+\cPhi_{j-1}^j+\cPhi_j^k+F(\bt).
   \end{align}
   The lemma now follows readily from the fact that $\cPhi_0^{j-1}+\cPhi_{j-1}^j+\cPhi_j^k\leq \lg \fkm^*,\fkw^* \rg$.
  \end{proof}
  
  Now, from \eqref{cg3}, \eqref{maxSs}, Lemma \ref{Lempsij}, and the  Borel-Cantelli Lemma, we get that
  \begin{align}
   \limsup_{N\uparrow \infty}\frac{1}{N}\log M_{(N)} \leq \lg \fkm^*,\fkw^* \rg+F(\bt)+2\frac{k}{L}\bs \bt \quad \pas
  \end{align}
  Replacing this in \eqref{cg21}, and since $L$ is arbitrary, 
  Proposition \ref{PropfkM1} follows.

 
\section{Proof of Proposition \ref{PropfkM2}}\label{Sec-PropfkM2}
Recall~(\ref{x2n}). 
Let us start by describing briefly the strategy we use to prove Proposition \ref{PropfkM2}. In Proposition \ref{SetG}, we have showed that, for each pair of vertices $(\e,\u)$ such that $\dH_1(\e_1,\u_1)<\eps N_1$, the path connecting them in $\G_N$ has, with $\p$-probability $1$, the form $\g_{\e\u}=\g_{\e\o}\cup\g_{\o\u}$ for all large enough $N$, where the vertex $\o$, which we will refer to here as {\it the intermediate point} of the path $\g_{\e\u}$, is such that $\dH_1(\e_1,\o_1)\geq \eps N_1$ and $\eps N_1 \leq \dH_1(\o_1,\u_1)=\dH(\o,\u)\leq 2\eps N_1$. 
Keeping this in mind, since the summation in the right-hand side of \eqref{x2n} is over a set of self-avoiding paths $\g_{\e\u}$ that go through the edge $e$, we have that either $e\in \g_{\e\o}$, or $e\in\g_{\o\u}$. So, our plan is to proceed with the estimation of $\fkM_N^2$ by considering these two cases separately.

Recall 
\eqref{SetSgNeu}. Using the above arguments, we get that with $\p$-probability $1$,
\begin{align}
\fkM_N^{(2)}\leq \fkN_N'+\fkN_N'' 
\end{align}
for all large enough $N$ where
\begin{align}\label{cg220}
 \fkN_N'=\max_{e\in\Gd}\left\{\sum_{\g_{\e\u}\ni e} \exp\left(-\bt[\h(\e)+\h(\u)]\right)\1\{\exists \o\in \Sg_N^{\e\u};\g_{\e\u}=\g_{\e\o} \cup \g_{\o\u} \text{ and } \g_{\e\o}\ni e\} \right\}
\end{align}
and
\begin{align}\label{cg221}
 \fkN_N''=\max_{e\in\Gd}\left\{\sum_{\g_{\e\u}\ni e} \exp\left(-\bt[\h(\e)+\h(\u)]\right)\1\{\exists \o\in \Sg_N^{\e\u};\g_{\e\u}=\g_{\e\o} \cup \g_{\o\u} \text{ and } \g_{\o\u}\ni e\}\right\}.
\end{align}

Notice that, by our construction, $\g_{\e\o}$ and $\g_{\o\u}$ have no edge in common.

Let us first estimate the term $\fkN_N''$. To do this, it is enough to notice that for a given edge $e=(\s,\t)$ the sum in the right-hand side of \eqref{cg221} is over a set of paths connecting pairs of vertices $(\e,\u)$ such that $\e$ is in a hypercube of dimension at most $N$ around $\s$ and $\u$ is in a hypercube of dimension at most $2\eps N_1$ around $\t$. Using this, it follows that
\begin{align}\label{cg23}
 \fkN_N''\leq 4^{\eps N_1} \exp\left(-\h(\bsg)\right)Z_N
\end{align}
hence, by Theorem 1.5(iii) in \cite{Bovier1} and \eqref{FrEn},
\begin{align}\label{cg24}
 \limsup_{N\uparrow\infty}\frac{1}{N}\log \fkN_N'' \leq F(\bt)+\lg \fkm^*,\fkw^*\rg \quad \pas
\end{align}
since $0<\eps<\nicefrac{1}{2}$ is arbitrary.

To estimate the term $\fkN_N^1$, we use basically the same argument that we have applied to prove Proposition \ref{PropfkM1}. Arguing as we did to get \eqref{cg21}, we can write
\begin{align}
 \fkN_N'\leq N \max_{1\leq i\leq N} \fkN_N'(i)
\end{align}
where $\fkN_N'(i)$ is as in \eqref{cg220} but with the paths $\g_{\e\o}$ in $\G^i$. Again, it is sufficient to consider the variable $\fkN_N'(1)$. Now, using the fact that the set $\{(\e,\o)\in\Sg_N\times \Sg_N^{\e,\u}: \g_{\e\u}^1\ni e\}$ is equal to
\begin{align}
 \left(\bigcup_{\e\in\{-1,+1\}^{i-1}} \{(\e,\s^i,\ldots,\s^N)\} \right)\times \left(\bigcup_{\o\in\{-1,+1\}^{N-i}}\{(\t^1,\ldots,\t^i,\o)\}\right)
\end{align}
for a given edge $e=(\s,\t)$, with respective $i\in\{1,\ldots,N\}$; using the same notation used in \eqref{cg3}, 
we readily get 
\begin{align}
 \fkN_N'(1)\leq 4^{\eps N_1}\max_{1\leq i\leq N} &\max_{\s\in\Sg_N} S_{i-1}^{(1)}(\s^i,\s^{>i})S_{N-i}^{(1)}(\s^{<i},-\s^i)
\end{align}
where the power of 4 error factor arises due to condition $\dH_1(\o_1,\u_1)=\dH(\o,\u)\leq 2\eps N_1$. Arguing now as at the end of Section \ref{Sec-PropfkM1}, since $0<\eps<\nicefrac{1}{2}$ is arbitrary, we conclude that
\begin{align}
 \limsup_{N\uparrow\infty}\frac{1}{N}\log \fkN_N' \leq F(\bt)+\lg \fkm^*,\fkw^*\rg \quad \pas
\end{align}
Hence, the claim of Proposition \ref{PropfkM2} holds.


\appendix
\section{Appendix}\label{AP1}
\begin{lem}\label{L5}
Let $\fkw^*=(\fkw_1^*,\ldots,\fkw_k^*)$ be the point of $\R^k$ such that
\begin{align}
\fkw_j^*=\bt_{l}\sqrt{a_j},\text{ if }j\in\{J_{l-1}^*+1,\ldots,J_{l}^*\} \text{ for some }l=1,\ldots,l_k,
\end{align}
Then, $\fkw^*\in\Psi_k$ and 
\begin{align}
 \max_{x\in \Psi_k} \lg\fkm^*,x\rg=\lg\fkm^*,\fkw^*\rg. 
\end{align}
\end{lem}

\begin{proof}
The proof of Lemma \ref{AP1} is inspired by a one in~\cite{Dorlas} and has as key tool the Cauchy-Schwarz inequality. The fact that $\fkw^*\in\Psi_k$ is an immediate consequence of definition \eqref{defb} and assumptions $\sum_{j=1}^k a_j=\sum_{j=1}^k p_j=1$. Now, let $x\in\Psi_k$. By Cauchy-Schwarz inequality, for all $l\in\{1,\ldots,l_k\}$, we have
\begin{align}
 \lg \Pi_0^{J_{l}^*}x,\Pi_0^{J_{l}^*}\fkw^*\rg \leq \left\|\Pi_0^{J_{l}^*}x\right\|\cdot\left\|\Pi_0^{J_{l}^*}\fkw^*\right\| \leq \left\|\Pi_0^{J_{l}^*}\fkw^*\right\|\cdot \sqrt{P_{J_{l}^*}}.
\end{align}
Since $\left\|\Pi_0^{J_{l}^*}\fkw^*\right\|^2=P_{J_{l}^*}$, it follows that
\begin{align}
\lg \Pi_0^{J_{l}^*}x,\Pi_0^{J_{l}^*}\fkw^*\rg \leq \left\|\Pi_0^{J_{l}^*}\fkw^*\right\|^2 = \lg \Pi_0^{J_{l}^*}\fkw^*,\Pi_0^{J_{l}^*}\fkw^* \rg.
\end{align}
Hence,
\begin{align}
 0\leq \lg \Pi_0^{J_{l}^*}\fkw^*-\Pi_0^{J_{l}^*}x,\Pi_0^{J_{l}^*}\fkw^*\rg=\sum_{i=1}^{l}\sum_{j=J_{i-1}^*+1}^{J_i^*}\bt_i\sqrt{a_j}(\bt_i\sqrt{a_j}-x_j).
\end{align}

Set $y_{l}=\sum_{j=J_{l-1}^*+1}^{J_{l}^*}\bt_{l}\sqrt{a_j}(\bt_{l}\sqrt{a_j}-x_j)$, $l=1,\ldots,l_k$, and consider the numbers $\bt\bt_1^{-1}>\cdots >\bt\bt_{l_k}^{-1}>0$. From what we have just seen, the sequences $(y_{l})_{l=1}^{l_k}$ and $(\bt\bt_{l}^{-1})_{l=1}^{l_k}$ satisfy the conditions of Lemma A in~\cite{Dorlas} so that we readily get
\begin{align}\label{iL5}
0\leq \sum_{l=1}^{l_k}\bt\bt_{l}^{-1}\sum_{j=J_{l-1}^*+1}^{J_{l}^*}\bt_{l}\sqrt{a_j}(\bt_{l}\sqrt{a_j}-x_j)=\lg \fkm^*,\fkw^* \rg - \lg \fkm^*,x \rg.
\end{align}
This concludes the proof of Lemma \ref{L5}.
\end{proof}

\begin{teo}[Projection onto a closed convex set]\label{PBrezis}
Let $K\subset H$ be a nonempty closed convex set. Then for every $f\in H$ there exists a unique element $u\in K$ such that
\begin{align}
|f-u|=\min_{v\in K}|f-v|=\d(f,K).
\end{align}
Moreover, $u$ is characterized by the property
\begin{align}
u\in K \text{ and } \lg f-u,v-u \rg \leq 0,\quad \forall v\in K.
\end{align}
\end{teo}

See~\cite{brezis1983analyse}, Theorem V.2, p. 79.

\section*{Acknowledgements}
This work is part of the Ph.D.~thesis of the second author
at IME-USP and was supported in part by CNPq 140762/2016-7. 
We warmfully thank Pierre Picco for suggesting this problem 
and for innumerable discussions concerning it in many occasions.

%




\end{document}